\documentclass{amsart}

\usepackage{amsmath}    
\usepackage{graphicx}   
\usepackage{verbatim}   
\usepackage{amssymb}
\usepackage{amsthm}

\usepackage{mathrsfs}
\usepackage{bbm}
\usepackage{exscale,color}
\usepackage{enumerate}

\newtheorem{theorem}{Theorem}

\newtheorem{corollary}[theorem]{Corollary}
\newtheorem{proposition}[theorem]{Proposition}

\theoremstyle{remark}

\theoremstyle{definition}
\newtheorem{definition}[theorem]{Definition}

\newcommand{\R}{\mathbb{R}}

\newcommand{\N}{\mathbb{N}}
\newcommand{\Z}{\mathbb{Z}}

\newcommand{\sinc}{\textnormal{sinc}}

\newcommand{\sgn}{\textnormal{sgn}}

\title[Gibbs--Wilbraham Phenomenon for Sampling Series]{On the Gibbs--Wilbraham Phenomenon for Sampling and Interpolatory Series}

\author{Keaton Hamm}
\address{Department of Mathematics, University of Arizona, Tucson, AZ, 85721, USA\footnote{The work for this article was done when the author was an Assistant Professor of Mathematics at Vanderbilt University. In the latter stages, the author acknowledges support from the NSF TRIPODS program, grant number CCF--1423411.}}
\email{hamm@math.arizona.edu}

\dedicatory{To N. Sivakumar}

\begin{document}


\begin{abstract}
The Gibbs--Wilbraham phenomenon for generalized sampling series, and related interpolation series arising from cardinal functions is investigated. We prove existence of the overshoot characteristic of the phenomenon for certain cardinal functions, and characterize existence of an overshoot for sampling series.
\end{abstract}



\maketitle

\section{Introduction}

It has long been observed that applying some smooth approximation method to functions with jump discontinuities leads to an overshoot phenomenon.  This observation for truncated Fourier series expansions of a periodic function is attributed to J. W. Gibbs owing to his short notes appearing in Nature in 1898 and 1899 \cite{Gibbs1,Gibbs2}.  However, as noted by Hewitt and Hewitt \cite{HewittGibbs}, this phenomenon was known to H. Wilbraham in 1848 \cite{Wilbraham}.  The interested reader is urged to consult their article, as it gives a fascinating description of the history of the phenomenon.
 
Gibbs' observation was precisely that despite the fact that the partial sums of the Fourier series of a periodic function converge pointwise to the function (or to the average value $\frac12(f(t+)+f(t-))$ across a jump discontinuity), the graph of the limit (i.e. of $f$), is not the (visual) limit of the graphs (of $S_N[f]$).  Stated another way, the set of limit points of the convergent partial sums lies outside of the interval determined by the jump discontinuity of the function.

Since then, the phenomenon has been explored for many other approximation methods, including spline interpolation \cite{FosterRichards,Richards} and wavelet expansions \cite{AtreasGibbs1,AtreasGibbs2, HuangGibbs, GibbsACHA,Shim}.  For more examples and a survey of the literature, see \cite{Jerri}.  The purpose of this short note is to examine the existence of a Gibbs--Wilbraham phenomenon for generalized sampling series, and related interpolation series arising from cardinal interpolants.  These are related to wavelet expansions, but also to approximants by shifts of positive definite functions, which include radial basis functions, for example.

\section{Generalized Sampling Series}

A function $\phi:\R\to\R$ is called a {\em generalized sampling kernel} (after Butzer, Ries, and Stens \cite{BRS}) provided $\phi\in C(\R)$, and the series $\sum_{n\in\Z}|\phi(t-n)|$ is uniformly convergent on $[0,1]$.  Given a generalized sampling kernel and a bounded function $f:\R\to\R$, the {\em sampling expansion} of $f$ is
\[S_W^\phi[f](t):=\sum_{n\in\Z}f\left(\frac nW\right)\phi(Wt-n),\quad t\in\R.\]

It is well-known that if $f$ has a compactly supported Fourier transform and $\phi$ is the cardinal sine function: $\sinc(x) = \sin(\pi x)/(\pi x)$ when $x\neq0$ and $\sinc(0)=1$, then $S_W^\sinc[f]=f$ for sufficiently large $W$ (this is the content of the Whittaker--Kotel'nikov--Shannon Sampling Theorem, \cite{Shannon}).

A characterization of Butzer, Ries, and Stens gives the following beautiful theorem.

\begin{theorem}[\cite{BRS}, Theorem 1]\label{THMBRS}
Suppose that $\phi:\R\to\R$ is a generalized sampling kernel.  Then the following are equivalent:
\begin{enumerate}[(i)]
\item $\displaystyle \sum_{n\in\Z}\phi(t-n)=1,\quad x\in[0,1)$.
\item For every bounded function $f:\R\to\R$, \[\underset{W\to\infty}\lim\; S_W^\phi[f](t)=f(t)\] for each $t\in\R$ which is a point of continuity of $f$.
\end{enumerate}
\end{theorem}

It should be noted that such sampling expansions are intimately related with the theories of reproducing kernel Hilbert spaces and shift-invariant spaces, but we will not dwell on this connection at the present moment.

In what follows, given $t\in\R$, we set $\displaystyle f(t-):=\lim_{x\to t^-}f(x)$ and $\displaystyle f(t+):=\lim_{x\to t^+}f(x)$.

\section{The Gibbs--Wilbraham Phenomenon}\label{SECGibbs}

The Gibbs--Wilbraham phenomenon has been well-studied for quite some time, and its influence is ubiquitous in smooth interpolation and approximation schemes targeting functions with jump discontinuities or cusps.  Its general description may be given as follows: suppose that $f$ is the pointwise limit of a convergent process, which we will denote by $T_N[f]$.  A Gibb's phenomenon is exhibited by this process provided that the set of its limit points is outside the range of $f$ itself.  In particular, suppose that $f(t-)< f(t+)$, then a Gibb's phenomenon is exhibited at $t$ provided$\{T_N[f](t+ \frac{\xi}{N}):N\in\N,\; \xi\neq t\}\supsetneq [f(t-),f(t+)]$.	As a particular example, if $T_N[f]$ is the trigonometric series $\sum_{n=-N}^N c_ne^{2\pi int}$ which best approximates the function $f(x) = \sgn(x)$ in $L_2[-\frac12,\frac12]$, then 
\[
\lim_{N\to\infty} T_N[f]\left(\frac \xi N\right) = 2 \int_0^\xi \frac{\sin(\pi x)}{\pi x}\;dx.
\]
Setting $\xi=1$ gives the absolute maximum of this function, which is approximately $1.17898$, and similarly setting $\xi=-1$ gives the absolute minimum.  Notice then that the set of limit points of $T_N[f]$ is the interval $[-1.17898,1.17898]$ which is strictly larger than the range of the jump discontinuity of $f$ at 0, which is $[-1,1]$.  One fascinating aspect of this phenomenon is that indeed this overshoot turns out to be dependent only on the magnitude of the jump discontinuity, but otherwise independent of the function $f$. 

For least-squares approximation as is described here, the pointwise values of $f$ are not so important, but if we consider generalized sampling series as above, evidently changing the values of $f$ could drastically alter the behavior of the sampling series.  To illustrate the Gibbs--Wilbraham phenomenon, we will suppose that $f$ has a jump discontinuity at $0$, and that $f(0)=f(0+)$.  One does get different behavior if one allows $f(0)$ to be some arbitrary number (e.g. \cite{AtreasGibbs2}), but our assumptions are sufficient to demonstrate existence of the Gibbs--Wilbraham phenomenon, and we do not intend to discuss how to mitigate its effects here.

\section{The General Gibbs--Wilbraham Function}

Given a generalized sampling kernel and a bounded function $f$ with a jump discontinuity at 0, satisfying $f(0)=f(0+)$, define the {\em Gibbs--Wilbraham function} associated with $f$ via
\begin{equation}\label{EQGibbs}
G_\phi[f](t):=f(0+)\sum_{n\geq0}\phi(t-n)+f(0-)\sum_{n<0}\phi(t-n).
\end{equation}
The reason for the terminology is that this is precisely the function which will classify the overshoot which is the characteristic feature of the Gibbs--Wilbraham phenomenon.  Indeed, considering the observation made in the previous section, we must consider the limit points of the sampling series, and hence for $t$ in a neighborhood of the origin, we are forced to consider
\[ S_N^\phi[f]\left(\frac tN\right) = \sum_{n\in\Z}f\left(\frac nN\right)\phi(t-n),\]
which as $N\to\infty$, converges pointwise and indeed uniformly to $G_\phi[f](t)$ as defined above (this follows from the boundedness of $f$, the fact that $\sum_{n\in\Z}|\phi(t-n)|$ is uniformly convergent, and the monotone convergence theorem, for example). 

\subsection{Properties of the Gibbs--Wilbraham Function}

Before proceeding, we pause to collect some basic facts about the Gibbs--Wilbraham functions.  

\begin{proposition}\label{PROPGibbsProperties}
Suppose $\phi$ is a generalized sampling kernel which satisfies one of the equivalent conditions of Theorem \ref{THMBRS}, $f$ and $g$ are bounded functions on $\R$, and $c\in\R$.  Then the following hold:
\begin{enumerate}[(i)]
\item $\displaystyle G_\phi[f\pm g] = G_\phi[f]\pm G_\phi[g];$
\item  $G_\phi[c]=c$;
\item $G_\phi[f+c] = G_\phi[f]+c.$
\end{enumerate}
\end{proposition}

The proof of the above Proposition is evident from the definition in \eqref{EQGibbs} and the fact that $\sum_{n\in\Z}\phi(x-n)=1$.

\subsection{Gibbs--Wilbraham Phenomenon}

With these considerations in mind, we make the following evident definition (cf. \cite{AtreasGibbs1}).

\begin{definition}
Suppose $\phi$ is a generalized sampling kernel which satisfies one of the equivalent conditions of Theorem \ref{THMBRS}.
\begin{itemize}
\item  If $f$ is a bounded function such that $f(0-)<f(0+)$, then $\phi$ exhibits a {\em left (resp. right) Gibbs--Wilbraham phenomenon} for $f$ provided there exists a $y<0$ (resp. an $x>0$) such that $G_\phi[f](y)<f(0-)$ (resp. $G_\phi[f](x)>f(0+)$).

\item The kernel $\phi$ exhibits a {\em strong Gibbs--Wilbraham phenomenon} for $f$ provided it exhibits both a left and a right Gibbs--Wilbraham phenomenon for $f$.

\item The kernel $\phi$ exhibits a left (resp. right, resp. strong) Gibbs--Wilbraham phenomenon provided it exhibits a left (resp. right, resp. strong) Gibbs--Wilbraham phenomenon for all bounded functions $f$.
\end{itemize}
\end{definition}

First, we note that the restriction that $f(0-)<f(0+)$ is no restriction at all.  Indeed, if $f(0-)>f(0+)$, we say that $\phi$ exhibits a left (respectively, right) Gibbs--Wilbraham phenomenon for $f$ if and only if $\phi$ exhibits a left (respectively, right) Gibbs--Wilbraham phenomenon for $-f$.

Secondly, we note that in reality, it suffices to consider the case when $f(0-)=-1$ and $f(0+)=1$, and thus in principle the case when $f$ is the unit step function $-\chi_{(-\infty,0)}+\chi_{[0,\infty)}$.  Indeed, suppose that $f(0-)<f(0+)$, and let $c\in\R$ be such that $f(0-)+c=-[f(0+)+c]$, and let $d\in\R$ be such that $df(0-)+dc = -1$, and hence $-d[f(0+)+c]=1$.  Then by Proposition \ref{PROPGibbsProperties}, $\phi$ exhibits a Gibbs--Wilbraham phenomenon for $f$ if and only if  it exhibits the phenomenon for $df+c$.  Thus we content ourselves with considering the general Gibbs--Wilbraham function 
\begin{equation}\label{EQGibbsReduction}
G_\phi(t):=\sum_{n\geq0}\phi(t-n)-\sum_{n<0}\phi(t-n).
\end{equation}

With these notions in hand, we may characterize the existence of a Gibbs--Wilbraham phenomenon for sampling series (we note that this theorem is essentially contained in \cite{AtreasGibbs1,Shim}, but our assumptions on $\phi$ above are more relaxed than in the latter, and we do not require $\phi$ to be the sampling function arising from a wavelet scaling function as in the former -- examples of this will be given in Section \ref{SECCardinal}).

\begin{theorem}\label{THMGibbsCharacterization}
Suppose $\phi$ is a generalized sampling kernel which satisfies one of the equivalent conditions of Theorem \ref{THMBRS}.  Then 
\begin{enumerate}[(i)]
\item $\phi$ exhibits a left Gibbs--Wilbraham phenomenon if and only if there exists a $y<0$ such that $\sum_{n\geq0}\phi(y-n)<0$.
\item  $\phi$ exhibits a right Gibbs--Wilbraham phenomenon if and only if there exists an $x>0$ such that $\sum_{n<0}\phi(x-n)<0$.
\item $\phi$ exhibits a strong Gibbs--Wilbraham phenomenon if and only if there exists $y<0$ and $x>0$ such that $\sum_{n\geq0}\phi(y-n)<0$ and $\sum_{n<0}\phi(x-n)<0$.
\end{enumerate}
\end{theorem}

\begin{proof}
For the proof of (i), note that $\phi$ exhibits a left Gibbs--Wilbraham phenomenon if and only if there exists a $y<0$ such that $G_\phi(y)<-1$.  By \eqref{EQGibbsReduction}, this is equivalent to
\[ \sum_{n\geq0}\phi(y-n)-\sum_{n<0}\phi(y-n) < -1\]
which is equivalent to
\[\sum_{n\geq0}\phi(y-n)-\left[\sum_{n<0}-\sum_{n\in\Z}\right]\phi(y-n) < 0.\]
Finally, this is equivalent to $2\sum_{n\geq0}\phi(y-n)<0$, which yields the desired conclusion.

The proof of (ii) follows by very similar reasoning, and so is omitted.  Item (iii) follows via combining (i) and (ii).
\end{proof}

\begin{corollary}\label{COREven}
Suppose $\phi$ is a generalized sampling kernel which satisfies one of the equivalent conditions in Theorem \ref{THMBRS}, and that $\phi$ is even.  Then the following are equivalent:
\begin{enumerate}[(i)]
\item $\phi$ exhibits a left Gibbs--Wilbraham phenomenon,
\item $\phi$ exhibits a right Gibbs--Wilbraham phenomenon,
\item $\phi$ exhibits a strong Gibbs--Wilbraham phenomenon.
\end{enumerate}
\end{corollary}
\begin{proof}
(ii)$\Rightarrow$(i): Suppose $x>0$ is the point exhibiting the right Gibbs--Wilbraham phenomenon.  Then letting $y=-x-1<0$, we have
\[\sum_{n\geq0}\phi(y-n) = \sum_{n\geq0}\phi(-x-n-1) = \sum_{n\geq0}\phi(x+n+1) = \sum_{n<0}\phi(x-n)<0,\]
where the final inequality comes from Theorem \ref{THMGibbsCharacterization}.  Consequently, the point $y=-x-1$ exhibits the left Gibbs--Wilbraham phenomenon.

Note that since (ii) implies (i), it follows that (ii) implies (iii) since (iii) is equivalent to  (i)$+$(ii). Finally, (iii) implies (ii) by definition, hence the proof is complete.
\end{proof}

To conclude, let us remark that if $\phi$ is even, and either for every $x>0$, $\sum_{n\geq0}\phi(x-n)=1$, or for every $y<0$, $\sum_{n<0}\phi(y-n)=1$, then $\phi$ does not exhibit a Gibbs--Wilbraham phenomenon (see Theorem \ref{THMGibbsCharacterization}).  One simple example of this is the $B$--spline of order 2 given by $M_2:=\chi_{[-1/2,1/2]}\ast\chi_{[-1/2,1/2]}$.  Its compact support is what forces the condition that for all $x>0$, $\sum_{n\geq0}\phi(x-n)=1$, and hence no Gibbs--Wilbraham phenomenon exists for this function (see also \cite[Example 1]{Shim}).  On the other hand, it does provide a convergent sampling expansion as in Theorem \ref{THMBRS} \cite[Corollary 3]{BRS}.

\section{Cardinal Functions and Interpolation}\label{SECCardinal}

It was shown in \cite{BRS} that recovery of average values of functions at jump discontinuities by their generalized sampling series is incompatible with enforcing an interpolation condition.  That is, if one requires that $S_W^\phi[f](t)\to\alpha f(t+)+(1-\alpha)f(t-)$ for some $\alpha\in\R$ whenever $f$ has a jump discontinuity at $t\neq0$, then it follows that $\phi(0)$ must be $0$, and hence that $S_W^\phi[f]\left(\frac{k}{W}\right)\neq f\left(\frac{k}{W}\right)$ for all $k\in\Z$.  However, this need not be true if we allow $t=0$.  In this case, if $\phi$ satisfies the interpolatory condition $\phi(k)=\delta_{0,k}$, $k\in\Z$, then we evidently have $S_W^\phi[f]\left(\frac{k}{W}\right)=f\left(\frac{k}{W}\right)$ for all $W>0$ and all $k\in\Z$.  

If $\phi$ satisfies the integer interpolatory condition above, then it is often called a {\em cardinal function}.  One particular method of manufacturing cardinal functions is to define them via their Fourier transforms in the following manner: suppose $\psi$ is given, and let
\[\widehat{L_\psi}(\xi):=\dfrac{\widehat\psi(\xi)}{\underset{k\in\Z}\sum \widehat\psi(\xi-k)}.\]
If, for instance, the series in the denominator is bounded away from zero on $[0,1]$, and $\widehat\psi\in L_1(\R)$, then $L_\psi$ defined by the Fourier inversion formula is indeed a cardinal function \cite{Buhmann90}. 

We now turn to explore the Gibbs--Wilbraham phenomenon for certain cardinal functions which have been studied, for instance, in the radial basis function interpolation literature.  We will assume that $L_\psi$, for a given choice of $\psi$, is a generalized sampling kernel which is additionally a partition of unity as in Theorem \ref{THMBRS} item (i).  Some examples include the cardinal function associated with the Hardy multiquadric $\sqrt{x^2+1}$, and the Poisson kernel $(x^2+1)^{-1}$, \cite{Buhmann90}.  It also follows from \cite[Section 4]{HL} that the cardinal functions for generalized multiquadrics, $(x^2+1)^\alpha$ for any $\alpha\in(-\infty,-1]\cup[1/2,\infty)$, satisfy these conditions.  The impetus for analyzing such cardinal functions was Schoenberg's analysis of cardinal $B$--splines \cite{SchoenbergBook}.  Other examples are the radial powers $|x|^{2k+1}$ for any $k\in\N\cup\{0\}$ and the thin-plate splines $|x|^{2k}\ln|x|$, for $k\in\N$ \cite{Buhmann90}.

Many of these examples stem from radial basis functions, but the construction is more general.  Some numerical evidence for the existence of the Gibbs--Wilbraham phenomenon for radial basis function interpolation was given in \cite{FornbergFlyerGibbs}.  

Let us note first that if $\widehat{\psi}$ is even, then so is $\widehat{L_\psi}$, and consequently $L_\psi$.  In this case, Corollary \ref{COREven} implies that a strong Gibbs--Wilbraham phenomenon is exhibited by the given cardinal function provided either a left or right one can be shown.  

In many instances, families of generating functions indexed by a given parameter are considered, i.e. $(\psi_\alpha)_{\alpha\in A}$.  For example, the $n$-th order $B$--spline given by $M_n:=\chi_{[-1/2,1/2]}\ast\dots\ast\chi_{[-1/2,1/2]}$, where there are $n$ terms in the convolution.  Additionally, families of multiquadrics have been used: $\phi_c(x)=(x^2+c^2)^\alpha$, $c>0$ for a fixed $\alpha$ in the range prescribed above.  It is known in many cases that asymptotically in the parameter (e.g. as $n\to\infty$, $c\to\infty$, or $\alpha\to-\infty$) the cardinal functions converge to the classical sinc function uniformly.  In fact, Ledford \cite{LedfordCardinal} gives sufficient conditions on a one-parameter family of functions $(\phi_{\alpha})_{\alpha\in A}$, which he calls {\em regular families of cardinal interpolators}, for which the associated cardinal functions $L_{\phi_\alpha}$ converge to sinc uniformly (where the parameter $\alpha$ has a natural limiting value, which is typically forced to be $\infty$ by convention).  Let us now discuss the Gibbs--Wilbraham phenomenon for such families, beginning with the following.

\begin{proposition}\label{PROPEven}
Suppose $\phi$ is even, then the associated Gibbs--Wilbraham function satisfies
\[ G_{\phi}\left(\frac12\right) = 2\phi\left(\frac12\right).\]
\end{proposition}

\begin{proof}
Putting $t=1/2$ into \eqref{EQGibbsReduction} yields
\begin{align*}G_{\phi}\left(\frac12\right) & = \phi\left(\frac12\right) + \sum_{n=1}^\infty\phi\left(\frac12-n\right)-\sum_{n=1}^\infty\phi\left(\frac12+n\right)\\
& = \phi\left(\frac12\right)+\phi\left(-\frac12\right) + \sum_{n=1}^\infty\phi\left(\frac12-(n+1)\right)-\sum_{n=1}^\infty\phi\left(\frac12+n\right)\\
& = 2\phi\left(\frac12\right),
\end{align*}
where evenness of $\phi$ was used in the final step.
\end{proof}

\begin{corollary}\label{CORRegularInterpolators}
Suppose that $(\phi_\alpha)_{\alpha\in A}$ is a regular family of cardinal interpolators such that $\phi_\alpha$ is even for every $\alpha$, and is a generalized sampling kernel satisfying one of the equivalent conditions of Theorem \ref{THMBRS}.  Then for sufficiently large $\alpha$, $L_{\phi_\alpha}$ exhibits a strong Gibbs--Wilbraham phenomenon.
\end{corollary}

\begin{proof}
From \cite[Proposition 2]{LedfordCardinal}, $L_{\phi_\alpha}\to\sinc$ uniformly, so appealing to Proposition \ref{PROPEven}, if suffices to notice that $2\sinc(1/2)=4/\pi>1$.
\end{proof}

Some examples of families satisfying the conditions of Corollary \ref{CORRegularInterpolators} are the generalized multiquadrics $(\phi_c)_{c\geq1}$ for a fixed $\alpha\in(-\infty,-3/2]\cup[1/2,\infty)\setminus\N$ (however, if one appeals to the more specific analysis of \cite{HL}, one finds that the permissible range of $\alpha$ may be extended to $(-\infty,-1]\cup[1/2,\infty)\setminus\N$).  Additionally, if one considers multiquadrics $\phi_\alpha(x)=(x^2+1)^\alpha$ and allows $\alpha$ to vary, then $(\phi_\alpha)_{\alpha\leq-1}$ and $(\phi_{\alpha_j})_{j\in\N}$ where $(\alpha_j)\subset[1/2,\infty)$ is unbounded with $\text{dist}(\{\alpha_j\},\N)>0$ yield cardinal functions which exhibit a strong Gibbs--Wilbraham phenomenon for large parameter. 

Moreover, the $B$--splines of order $n$, with $n$ tending toward $\infty$ provide another example of a family exhibiting a strong Gibbs--Wilbraham phenomenon.  That $L_{M_n}$ satisfies condition (i) of Theorem \ref{THMBRS} is known (it follows easily from the fact that $\widehat{M_n}(\xi)=\sinc^n(\xi)$ and the Poisson Summation formula).  Additionally, we find from \cite[Lecture 9]{SchoenbergBook} that $L_{M_n}\to\sinc$ uniformly as $n\to\infty$.

While the above corollary only demonstrates the overshoot phenomenon for large parameter, numerical experiments reveal that even for small shape parameter $c\geq1$ in the multiquadrics, for example, the cardinal function is quite close to the corresponding value of sinc.  See also \cite{FornbergFlyerGibbs} for illustrations of this fact.

Let us mention that one important family used in cardinal interpolation is not covered here: namely, the Gaussians $(e^{-|x/\alpha|^2})_{\alpha\geq1}$.  The family of cardinal functions associated with the Gaussian is well-known to provide recovery of Paley--Wiener (or bandlimited) functions (e.g. \cite[Theorem 3.7]{BaxterSiva}).  However, the cardinal functions do not satisfy condition (i) of Theorem \ref{THMBRS}, \cite{Buhmann90}.  Thus, the sampling expansion related to the Gaussian does not give convergence at all points of continuity of bounded functions, and hence speaking of a Gibbs--Wilbraham phenomenon in this case does not precisely make sense (cf. our remarks in Section \ref{SECGibbs}).   

To conclude, we note that in some instances if the function $\psi$ decays sufficiently fast and the symbol $\sum_{k\in\Z}\widehat{\psi}(\xi-k)$ and its reciprocal are in the Wiener algebra of functions whose Fourier coefficients are summable, then the cardinal interpolant may be written in a different form.  Specifically, there are unique $\ell_\infty$ coefficients $(a_n)$ such that 
\[S_W^{L_\psi}[f] (t) = \sum_{n\in\Z}f\left(\frac nW\right)L_\psi(Wt-n) = \sum_{n\in\Z}a_n\psi(Wt-n),\]
where the equality and convergence of the series is in $L_\infty$.  The condition on the decay of $\psi$ is that $\sum_{n\in\Z}\|\psi(\cdot-n)\|_{L_\infty[0,1]}<\infty$, which is often written $\psi\in W(L_\infty,\ell_1)$, where the space defined thusly is Wiener's space.  The proof of this fact follows from the argument in \cite[Theorem 3.2]{HL2}, and the multiquadrics for negative exponent $\alpha$ are examples of functions satisfying these criteria.

\section{Summary}

In this brief note, we have discussed the characterization of a Gibbs--Wilbraham phenomenon for generalized sampling series.  For generalized sampling kernels, the existence of this phenomenon may be reduced to considering the canonical case of approximating the function $f(x)=\sgn(x)$ with the convention that $\sgn(0)=1$.  Moreover, existence of the phenomenon is purely determined by a series of translates of the sampling kernel $\phi$ (Theorem \ref{THMGibbsCharacterization}).  In contrast to previous works analyzing the Gibbs--Wilbraham phenomenon for wavelet sampling series, we have not required the sampling kernel to be the scaling function of a wavelet system.  Moreover, our considerations for cardinal functions arising from radial basis functions provides a different proof than that indicated by the numerical observations of \cite{FornbergFlyerGibbs}.  Additionally, we found that Ledford's conditions for regular families of cardinal interpolators yields a variety of examples of sampling kernels whose cardinal functions exhibit a Gibbs--Wilbraham phenomenon.

\bibliographystyle{plain}
\bibliography{GibbsBib}

\begin{thebibliography}

\bibitem{AtreasGibbs1} N.~Atreas and C.~Karanikas.
\newblock Gibbs phenomenon on sampling series based on {S}hannon's and
  {M}eyer's wavelet analysis.
\newblock {\em J. Fourier Anal. Appl.}, 5(6):575--588, 1999.

\bibitem{AtreasGibbs2}
N.~Atreas and C.~Karanikas.
\newblock Reducing gibbs ripples for some wavelet sampling series.
\newblock In A.~Jerri, editor, {\em Advances in the {G}ibbs Phenomenon},
  chapter~11, pages 335--362. Clarkson University $\Sigma$ Sampling Publishing,
  Potsdam, NY, 2007.

\bibitem{BaxterSiva}
B.~J.~C. Baxter and N~Sivakumar.
\newblock On shifted cardinal interpolation by gaussians and multiquadrics.
\newblock {\em J. Approx. Theory}, 87(1):36--59, 1996.

\bibitem{Buhmann90}
M.~D. Buhmann.
\newblock Multivariate cardinal interpolation with radial-basis functions.
\newblock {\em Constr. Approx.}, 6(3):225--255, 1990.

\bibitem{BRS}
P.~L. Butzer, S.~Ries, and R.L. Stens.
\newblock Approximation of continuous and discontinuous functions by
  generalized sampling series.
\newblock {\em J. Approx Theory}, 50(1):25--39, 1987.

\bibitem{FornbergFlyerGibbs}
B.~Fornberg and N.~Flyer.
\newblock The {G}ibbs phenomenon for radial basis functions.
\newblock In A.~Jerri, editor, {\em Advances in the {G}ibbs Phenomenon},
  chapter~6, pages 197--217. Clarkson University $\Sigma$ Sampling Publishing,
  Potsdam, NY, 2007.

\bibitem{FosterRichards}
J.~Foster and F.B. Richards.
\newblock {G}ibbs--{W}ilbraham splines.
\newblock {\em Constr. Approx.}, 11(1):37--52, 1995.

\bibitem{Gibbs1}
J.~W. Gibbs.
\newblock Fourier's series.
\newblock {\em Nature}, 59(1522):200, 1898.

\bibitem{Gibbs2}
J.~W. Gibbs.
\newblock Fourier's series.
\newblock {\em Nature}, 59(1539):606, 1899.

\bibitem{HL}
K.~Hamm and J.~Ledford.
\newblock Cardinal interpolation with general multiquadrics.
\newblock {\em Adv. Comput. Math.}, 42(5):1149--1186, 2016.

\bibitem{HL2}
K.~Hamm and J.~Ledford.
\newblock Cardinal interpolation with general multiquadrics: convergence rates.
\newblock {\em Adv. Comput. Math.}, 2017.
\newblock doi:10.1007/s10444-017-9578-0.

\bibitem{HewittGibbs}
E.~Hewitt and R.~E. Hewitt.
\newblock The {G}ibbs-{W}ilbraham phenomenon: an episode in {F}ourier analysis.
\newblock {\em Archive for History of Exact Sciences}, 21(2):129--160, 1979.

\bibitem{HuangGibbs}
D.~Huang and Z.~Zhang.
\newblock Asymptotic behavior of {G}ibbs functions for {M}--band wavelet
  expansions.
\newblock {\em Acta Mat. Sin. (Engl. Ser.)}, 15(2):165--172, 1999.

\bibitem{Jerri}
A.~Jerri, editor.
\newblock {\em Advances in The {G}ibbs Phenomenon}.
\newblock Clarkson University $\Sigma$ Sampling Publishing, Potsdam, NY, 2007.

\bibitem{GibbsACHA}
S.~E. Kelly.
\newblock Gibbs phenomenon for wavelets.
\newblock {\em Appl. Comput. Harmon. Anal.}, 3(1):72--81, 1996.

\bibitem{LedfordCardinal}
J.~Ledford.
\newblock On the convergence of regular families of cardinal interpolators.
\newblock {\em Adv. Comput. Math.}, 41(2):357--371, 2015.

\bibitem{Richards}
F.~B. Richards.
\newblock A {G}ibbs phenomenon for spline functions.
\newblock {\em J. Approx. Theory}, 66(3):334--351, 1991.

\bibitem{SchoenbergBook}
I.~J. Schoenberg.
\newblock {\em Cardinal spline interpolation}.
\newblock SIAM, 1973.

\bibitem{Shannon}
C.~E. Shannon.
\newblock Communication in the presence of noise.
\newblock {\em Proceedings of the IRE}, 37(1):10--21, 1949.

\bibitem{Shim}
G.~G. Walter and H.-T. Shim.
\newblock {G}ibbs' phenomenon for sampling series and what to do about it.
\newblock {\em J. {F}ourier Anal. Appl.}, 4(3):357--375, 1998.

\bibitem{Wilbraham}
H.~Wilbraham.
\newblock On a certain periodic function.
\newblock {\em Cambridge and Dublin Math. J}, 3(198):1848, 1848.

\end{thebibliography}

\end{document}